\documentclass[12pt,reqno]{amsart}
\usepackage{amssymb,amsmath}
\def\({\left(}
\def\){\right)}
\def\Nx{\nabla_x}

\def\Om{\Omega}

\def\R {\mathbb{R}}

\newcommand{\be}{\begin{equation} }
\newcommand{\ee}{\end{equation} }

\def \rw {\rightarrow}

\def \p {\partial}

\def \and{\qquad\text{and}\qquad}

\def\Dx{\Delta_x}
\def\({\left(}
\def\){\right)}
\def\Nx{\nabla}

\def\al{\alpha}
\def\Om{\Omega}

\def \pt {\partial_{t}}

\def\R {\mathbb{R}}

\def \p {\partial}

\def \and{\qquad\text{and}\qquad}

\def\Dx{\Delta}
\newtheorem{proposition}{Proposition}[section]
\newtheorem{theorem}[proposition]{Theorem}

\newtheorem{lemma}[proposition]{Lemma}
\theoremstyle{definition}

\newtheorem{remark}[proposition]{Remark}
\newtheorem{example}[proposition]{Example}
\numberwithin{equation}{section}
\def\be{\begin{equation}}
\def\ee{\end{equation}}
\def\bp{\begin{proof}}
\def\ep{\end{proof}}

\def \no#1#2#3 {{\bf #1} (#3), #2.}
\def \eds#1#2#3 {#1, #2, #3.}
\title[ About Blow up of Solutions  to  Nonlinear  Wave Equations ]
{About  Blow up of Solutions With Arbitrary Positive Initial Energy to  Nonlinear  Wave Equations }
\author[]
{B. A. Bilgin and V. K. Kalantarov }

\address{( B. A. Bilgin) Department of Mathematics,
\newline\indent Ko{\c c} University, Rumelifeneri Yolu, Sariyer, Istanbul, Turkey
}

\address{(V.K.Kalantarov) Department of Mathematics,
\newline\indent Ko{\c c} University, Rumelifeneri Yolu, Sariyer, Istanbul, Turkey\\
\newline\indent Institute of Mathematics and Mechanics, Academy of Sciences of Azerbaijan,
\newline\indent   B. Vahabzade Street 9, 1141 Baku, Azerbaijan
}

\begin{document}

\begin{abstract}{ We show that blow up of solutions with arbitrary positive initial energy of the Cauchy problem for the abstract wacve eqation of  the form $Pu_{tt}+Au=F(u) \ (*)$ in a Hilbert space, where $P,A$ are positive  linear operators and $F(\cdot)$ is a continuously differentiable gradient operator can be obtained from the result of H.A. Levine on the growth of solutions of the Cauchy problem for (*).  This result is applied to the study of  inital boundary value problems for nonlinear Klein-Gordon equations, generalized Boussinesq equations and  nonlinear plate equations. A result on blow up of solutions with positive initial energy of the initial boundary value problem for wave equation under nonlinear boundary condition is also obtained.
}
\end{abstract}


\maketitle

\bigskip

\section{Introduction}
We consider the following problem 
\be\label{doe1}
Pu_{tt}+Au=F(u),
\ee
\be\label{doe1a}
u(0)=u_0, \ \ u_t(0)=u_1
\ee
in a Hilbert space $H$ with the inner product $(\cdot ,\cdot )$ and
the corresponding norm $\left\| \cdot \right\|.$
 We denote by $u$ a vector-function with domain $\left[
0,T\right) $ and range $D,$ where  $D$ is a dense linear subspace of
$H$. Suppose that $P, A$ are symmetric, positive definite, linear operators
defined on $D$, $F(\cdot ): D\rw H$ is a nonlinear
gradient  operator defined on $ D$ with the potential 
$G(u): D\rw R.$
We assume also that
\be\label{FG}
(F(v),v)\geq 2(1+2\al)G(v)- 2R_0, \ \ \forall v\in D
\ee
for some $\al >0, R_0\geq 0$.\\
For the sake of simplicity it is  assumed that $u(t)$ is a strong
solution of \eqref{doe1}, i.e. a solution $u$ for which all terms in \eqref{doe1} are
elements of $L^2(0,T;H)$ and $u(\cdot), u_t(\cdot)\in C(0,T;H).$\\
The idea of the concavity method of H. A. Levine introduced
in \cite{LE1} is based on a construction of some positive functional
$\Psi (t)=\psi (u(t))$, which is defined in terms of the local
solution of the problem (the local solvability of the problem is
therefore required) and proving that the function $\Psi (t)$
satisfies the inequality \eqref{Lin} given in the following
statement:

\begin{lemma}\label{Le}(see \cite{LE1}) Let $\Psi (t)$ be a positive,
twice differentiable function ,which satisfies
the inequality
\begin{equation}\label{Lin}
\Psi''(t)\Psi (t)-(1+\alpha )\left[ \Psi'(t)\right] ^2\geq 0, \ \ t\geq t_0
\end{equation}
with some $\alpha >0.$  If $\Psi (t_0)>0$ and $\Psi'(t_0)>0,$ then there
exists a time $t_1\leq t_0 + \frac{\Psi (0)}{\alpha \Psi'(0)}$ such that
$\Psi (t)\rightarrow +\infty $ as $t\rightarrow t_1^-.$
\end{lemma}
The concavty method and its modifications was used in the study of various nonlinear partial differential equations (see e.g. 
\cite{AKS}, \cite{EEE}, \cite{KL},\cite{LE3},\cite{LE4}, \cite{Run},\cite{STR}, \cite{ST2}).\\

There is
a number of papers devoted to the question of blow up of solutions
to the Cauchy problem and initial boundary value problems
for nonlinear wave equations with arbitrary large initial energy.\\
One of  the first results of this type is the result of H. Levine and
G. Todorova \cite{LT},\\
The concavity method and its modifications is employed
 to find sufficient conditions of blow up of solutions to the Cauchy problem and initial boundary value problems 
 for nonlinear Klein - Gordon equation, damped Kirchhoff-type equation, generalized Boussinesq equaton, quasilinear strongky damed wave equations and some other equations
(see, e.g.\cite{BiKa}, \cite{GS}-- \cite{Kut2}, \cite{Mes}, \cite{PP},\cite{SoZh}, \cite{WO},
 \cite{ZMZ} and references therein).

Our aim is to show that blow up of solutions with arbitrary positive initial energy of the problem \eqref{doe1} actually can be established by using the Lemma \ref{Le} and the following theorem 
on growth of solutions of the problem obtained in \cite{LE1}.
\begin{theorem}\label{Levg}  (\cite{LE1}) Suppose that the  $P,A: D\rw H$ are positive symmetric operators, 
 $F(\cdot): D\rw H$ satisfies the condition \eqref{FG} and $u$ is a solution of the problem \eqref{doe1}, \eqref{doe1a}. Suppose that the initial data satisfy the conditions
\be\label{L1}
(u_0,Pu_1)/(u_0,Pu_0)>0,
\ee
\be\label{L2}
\frac12(u_0,Au_0)+\frac12 (Pu_1,u_1)-G(u_0)+\frac{R_0}{(1+2\al)}<\frac12(u_0,Pu_1)^2/(u_0,Pu_0).
\ee
Then
$$
\lim\limits_{t\rw +\infty}(u(t),Pu(t))=+\infty,
$$
if $u(\cdot)$ exists on $(0,+\infty)$.
\end{theorem}

\section{Blow Up of Solutions to Abstract Wave Equations.}
In this section we find sufficient conditions for finite-time blow up of solutions to the problem \eqref{doe1}, \eqref{doe1a}  when the initial energy may take arbitrary positive values.

\begin{theorem}\label{T1}
Suppose that the operators $P,A$ and $F$ satisfy all the
conditions of Theorem \ref{Levg}  and suppose that there exists $a_0>0$ such that
\be\label{AP}
(Av,v)\geq a_0 (Pv,v), \ \ \forall v\in D.
\ee
Then there exists $t_1>0$ such that
\be\label{blow1}
\lim\limits_{t\rw t_1^-}(Pu(t),u(t))=+\infty.
\ee
\end{theorem}
\begin{proof}
Assume that all solutions of the problem \eqref{doe1}, \eqref{doe1a} are global solutions, i.e. they are defined for all $t\in (0,+\infty).$
Thanks to the Theorem \ref{Levg} if $u$ is a solution of the problem \eqref{doe1}, \eqref{doe1a}, then
\be\label{inf}
\Psi(t):=(Pu(t),u(t)) \rw +\infty  \ \ \mbox{as} \ \ t\rw \infty.
\ee
On the other hand 
$$
\Psi^\prime(t)=2(Pu_t(t),u(t)), 
$$
$$
\Psi^{\prime\prime}(t)=2(Pu_t(t),u_t(t))+2(Pu_{tt}(t),u(t)).
$$
Employing the equation \eqref{doe1} and the condition \eqref{FG} we get
\begin{multline*}
\Psi^{\prime\prime}(t)=2(Pu_t(t),u_t(t))-2(Au(t),u(t))+2(F(u(t),u(t))\\
\geq 2(Pu_t(t),u_t(t))-2(Au(t),u(t))+4(1+2\al)G(u(t)-4R_0.
\end{multline*}
 As usual, we define the energy as
\be\label{En}
E(t) := \frac12 (Pu_t(t),u(t))+\frac12 (Au(t),u(t))-G(u(t)),
\ee
and find that 
\be\label{energy}
E(t)=E(0)=\frac12(u_0,Au_0)+\frac12 (Pu_1,u_1)-G(u_0), \ t>0.
\ee
By using the energy equality \eqref{energy} we obtain from the last inequality that
\begin{multline}\label{inf1}
\Psi^{\prime\prime}(t)\geq 4(1+2\al)\left[ -\frac12 (Pu_t(t),u_t(t))-\frac12 (Au(t),u(t))+ G(u(t))\right]\\+
4(\al +1) (Pu_t(t),u_t(t))+4\al (Au(t),u(t)) -4R_0\\ \geq -4(1+2\al)E(0)-4R_0+4(\al +1) (Pu_t(t),u_t(t))+4\al (Au(t),u(t)) .
\end{multline}
Thus, by using the Cauchy - Schwarz inequality and the condition \eqref{AP} we obtain
\begin{multline}\label{inf2}
\Psi''(t)\Psi (t)-(1+\alpha )\left[ \Psi'(t)\right] ^2
\geq -\left( 4(1+2\al)E(0)+4R_0\right)\Psi(t)+4\al a_0\Psi^2(t) \\
+4(\al +1)\left[ (Pu_t(t),u_t(t))(Pu(t),u(t)) -(Pu_t(t),u(t))^2\right]\\
\geq \left[\al a_0 \Psi(t)-4(1+2\al)E(0)-4R_0\right]\Psi(t).
\end{multline}
Thanks to the Theorem \ref{Levg} the function $\Psi(t)$ tends to $+\infty$ as $t\rw +\infty$. Therefore, there exists $t_*>0$ such that
$$
\al a_0 \Psi(t)-4(1+2\al)E(0)-4R_0 \geq \delta > 0, \ \ \forall t\geq t_*.
$$
Hence, \eqref{inf2} implies that
$$
\Psi''(t)\Psi (t)-(1+\alpha )\left[ \Psi'(t)\right] ^2
\geq 0, \ \ \forall t\geq t_*.
$$
Moreover, in view of the assumption on $t_*$ by using \eqref{AP} we easily deduce from \eqref{inf1} that
$$
\Psi''(t)\geq \delta>0 \ \ \forall t \geq t_*.
$$
Consequently, there exists some $t_0\geq t_*$ such that $\Psi'(t_0)>0$.
Now we can apply the Lemma \ref{Le} and deduce that 
$$
(Pu(t),u(t))\rw +\infty, \ \ as t\rw t_1^-.
$$
\end{proof}

\section{Examples of Nonlinear Wave equations}
\noindent {\bf  1. Nonlinear Klein-Gordon Equation}
Let $u$ be a local strong solution to the Cauchy problem
\be\label{kg1}
\begin{cases}
\pt^2 u-\Dx u +m^2u= |u|^2u , \ x \in \R^3, t>0,\\
u(x,0)=u_0(x),  \ \pt u(x,0)=u_1(x),
\end{cases}
\ee where $m>0$ is a  given number, $u_0\in H^1(\R^n), u_1\in
L^2(\R^n)$ are given compactly supported functions.\\
The equation can be written in the form \eqref{doe1} with $P=I, \ A=
-\Dx +m^2I$ and $F(u)=|u|^2u$.
It follows from Theorem 
\ref{T1} that if \be\label{kgc} 
(u_0,u_1)>\left[\|u_1\|^2+\|\Nx
u_{0}\|^2+m^2\|u_0\|^2-\frac12\int_{\R}|u_0(x)|^{4}dx\right]^\frac12 \|u_0\| ,\ee
then the solution of the problem \eqref{kg1} blows up in a finite
time. If $u_0$ is a smooth nonnegative, nontrivial compactly
supported function then for    $u_1=\frac1{\sqrt{2}}u_0^{2}$ the
initial energy is
$$
E(0)=\frac12\|\Nx u_{0}\|^2+\frac{m^2}2\|u_0\|^2
$$
and the condition \eqref{kgc} takes the form \be\label{kgc1}
\int_{\R}|u_0(x)|^{3}> \sqrt{2} \left[ \|\Nx u_{0}\|^2+m^2\|u_0\|^2\right]^\frac12 \|u_0\|. \ee It is clear
that there is a wide class of functions $u_0$ for which the energy
takes any large value and the condition \eqref{kgc1} holds true.
\begin{remark} The Theorem \eqref{T1} holds true also for solutions of  the initial boundary
value problem for the nonlinear wave  equation under the homogeneous Dirichlet boundary  condition:
\be\label{w1}
\begin{cases}
\pt^2u-\Dx u +m^2u =|u|^pu, \ x \in \Om, t>0,\\
u(x,0)=u_0(x), \pt u(x,0)=u_1(x), \ x \in \Om, \\
u(x,t)=0, \ x \in
\p \Om, t>0,
\end{cases}
\ee
where $p$ is an arbitrary positive number if $n=1,2$ and $p\in (0,\frac{2}{n-2}]$ if $n\geq 3$.\\
Let us note that this result easily follows from the results of
 T. Cazenave obtained in \cite{CA} for solutions of the problem \eqref{w1} and the Theorem \ref{Levg} of H. A. Levine.\\
Indeed, T. Cazenave proved that each solution of the problem \eqref{w1} either  blows up in a finite time or is uniformly bounded.\\
Thus, if the functions $u_0,u_1$  satisfy the conditions of Theorem \ref{Levg}, that is
$$
(u_0,u_1)>\left[\|u_1\|^2+\|\Nx
u_{0}\|^2+m^2\|u_0\|^2-\frac{2}{p+2}\int_{\R}|u_0(x)|^{p+2}dx\right]^\frac12 \|u_0\| ,
$$
then the corresponding local solution of the
problem \eqref{w1} can not be continued on the whole interval $[0,\infty)$, i.e. it must blow up in a finite time.
\end{remark}

\begin{example}\label{Bsq} {\bf Generalized Boussinesq Equation}

Similarly we can find sufficient conditions for blow up of solutions with arbitrary positive initial energy for the generalized Boussinesq equation 
\be\label{G.Bsq.}
\pt^2 u-a\Delta u_{tt}-\Delta u +\nu \Delta^2 u+\Delta f(u)=0, \ x\in \Omega, t>0
\ee
under the homogeneous Dirichlet boundary conditions 
$$
u=\Delta u=0, \ x\in \p \Om, t>0,
$$ 
where $f(u)=|u|^mu+P_{m-1} (u), \ m\geq1$ is a given integer, $a\geq 0,\nu>0$ are given numbers,
$\Om\in \R^n$ is a bounded domain and $P_{m-1} (u)$ is a polynomial of order $\leq m-1.$
Applying $(-\Delta)^{-1}$ to \eqref{G.Bsq.},
where $-\Delta$ is the Laplace operator under Dirichlet boundary conditions,
we obtain an equation of the form \eqref{doe1} with
$P= (-\Delta)^{-1} + a I$, $A= I - \nu \Delta$ and $F$ replaced by $f$.
It is easy to see that there is $R_0\geq0$ such that $f$ satisfies
\eqref{FG} with $G(u)=\int_\Om\int_0^u f(s)dsdx$  and $\alpha= \frac{m}{4}$.
We consider initial data $(u_0,u_1)\in H^1_0(\Om)\times L^2(\Om)$.
Since $\Om$ is bounded, the Poincare inequality assures that the assumption \eqref{AP} is verified.
Hence, the conclusion of Theorem \ref{T1} holds provided that the assumptions of Theorem \ref{Levg} are fullfilled,
that is $u_0,u_1$ satisfy 
$$
(Pu_0,u_1)>0,
$$
$$
\frac12 \frac{(Pu_0,u_1)^2}{(Pu_0,u_0)}> E(u_0,u_1) + \frac{R_0}{1 + \frac{m}{2}} ,
$$
where
$$
E(u_0,u_1)= \frac12 (Pu_1,u_1)+ \frac12(Au_0,u_0) - \int_\Om \int_0^{u_0} f(s)dsdx.
$$
Let us prove that for any given number $K^2>0$ and any pair of functions\\
$[\hat{u}_0, \hat{u}_1 ]\in H^1_0(\Om)\times L^2(\Om)$ with
$$(P\hat{u}_0, \hat{u}_1 ) = \theta>0,  (P\hat{u}_0,\hat{u}_0)=(P\hat{u}_1,\hat{u}_1)=1,$$
there are uncountably infinitely many data of the form $u_i=c_i \hat{u}_i$, $c_i>0$, $i=0,1$,
such that $E(u_0,u_1)=K^2$ and the above conditions are satisified.
Here, note that necessarily $0<\theta\leq 1$.
Observe that in this case it is enough to verify only the latter of the two conditions above.
Rewriting this condition for the initial data of the form described above we find
$$
\frac12 c_1^2 \theta^2 >  E(c_0\hat{u}_0, c_1\hat{u}_1) + \frac{2R_0}{m+2} .
$$
Thus, the question is, given $K^2>0$, can we find $c_0,c_1>0$ so that this inequality is satisfied together
with the equality $ E(c_0\hat{u}_0, c_1\hat{u}_1)=K^2$.
So, let $K^2>0$ be given and 
fix any $c_1 > \theta^{-1}\left[2K^2 + \frac{4R_0}{m+2} \right]^\frac12.$
Note that, if for this fixed value of $c_1$ we can find $c_0>0$ such that $ E(c_0\hat{u}_0, c_1\hat{u}_1)=K^2$,
then the inequality condition above is automatically satisfied, and we are done.
It is easy to see that it is possible for any such $c_1$. Indeed, since $\theta \leq 1$, we have
$$
\frac12 c_1^2 > K^2,
$$
and consequently the continuous function
$$
H(c_0):= E(c_0\hat{u}_0, c_1\hat{u}_1) - K^2 = \frac12 c_1^2 - K^2 + \frac12 c_0^2(A\hat{u}_0,\hat{u}_0) - \int_\Om \int_0^{c_0\hat{u}_0} f(s)dsdx
$$
has the property
$$
\lim\limits_{c_0\rightarrow 0^+} H(c_0) = \frac12 c_1^2 - K^2 > 0.
$$
Moreover, by the structure of $f$ we have 
$$
\lim\limits_{c_0\rightarrow \infty } H(c_0) = - \infty.
$$
Due to  the intermediate value theorem we deduce that there is $c_0>0$ such that $H(c_0)=0$, and this finishes the proof.
\end{example}

\begin{example}\label{NPE} {\bf Nonlinear Plate  Equations}
It is clear that we can apply Theorem \ref{T1} to find sufficient conditions of blow up of solutions to intial boundary value problems  for the nonlinear plate equations of the form
$$
u_{tt}+\Delta^2u +\left(a_1+b_1 \int_\Om u_{x_1}^2dx\right)u_{x_1x_1}+\left(a_2+b_2 \int_\Om u_{x_2}^2dx\right)u_{x_2x_2}=0,  \ x\in \Om, t>0,
$$
and 
$$
u_{tt} +\Delta^2 u=f(u),  \ x \in \Om, t>0,
$$
under the boundary conditions
$$
u=\Delta u=0,  \ x\in \p \Om,
$$
where $f(\cdot): \R\rw \R$ is a continuous function which satisfies the condition
\be\label{Fc}
f(s)s-2(1+2\al) F(s)\geq -r_0,  \ \forall s\in \R, 
\ee
$r_0\geq 0, a_1,a_2\in \R , b_1>0,b_2>0$ are given numbers, $\ F(s):=\int_0^sf(\tau)d\tau$ and $\Om \subset \R^2$ is a bounded domain with sufficiently smooth boundary $\p \Om.$
\end{example}
\begin{remark}\label{R1}
Appliying Theorem \ref{T1}  we can obtain similar results on blow up of solutions to 
\begin{itemize}
\item initial boundary value problem for 
improved  Boussinesq equation
$$
u_{tt}-\Delta u_{tt} -\Delta u+\Delta^2 u+\Delta ^2 u_{tt} +\Delta (f(u))=0, 
$$
 where $f(\cdot)$ is a smooth function which satisfies  \eqref{Fc},\
\item  Cauchy problem and initial boundary value for system of nonlinear Klein-Gordon equation
$$
\begin{cases}
u_{tt}-\Delta u +m^2u =uv^2 +h_1(x), 
\\
v_{tt}-\Delta v +\mu^2v =vu^2 +h_2(x),
\end{cases}
$$
where $h_1,h_2\in L^2(\R^3)$ are given functions.
\end{itemize}
\end{remark}
\section{Second Order Wave Equation Under Nonlinear Boundary Conditions}
In this section we consider the following problem
\be\label{nb1}u_{tt}-\Delta u+bu=0, \ x\in \Om, t>0,
\ee
\be\label{nb2}
u(x,0)=u_0(x), \ \ u_t(x,0)=u_1(x), \ \ x\in \Om,
\ee
\be\label{nb3}
\frac{\p u}{\p n}=f(u), \ \ x\in \p \Om, t>0,
\ee
where $\Om\subset \R^N$ is a bounded domain with sufficiently smooth boundary, $b>0$ is a given number, $\vec{n} $ denotes the outward directed normal to $\p \Om$ and $f(\cdot):\R\rw\R$ is nonlinear term that satisfies the condition \eqref{Fc}.\\
The energy equality in this case has the form
\be\label{NBE}
E(t):=\frac12\|u_t(t)\|^2_{L^2(\Om)}+\frac12\|\nabla u(t)\|^2_{L^2(\Om)}+\frac b2\|u(t)\|^2_{L^2(\Om)} -\int\limits_{\p \Om}F(u(x,t))d\sigma=E_0,
\ee
where 
$$
E_0=\frac12\|u_1\|^2_{L^2(\Om)}+\frac12\|\nabla u_0\|^2_{L^2(\Om)}+\frac b2\|u_0\|^2_{L^2(\Om)} -\int\limits_{\p \Om}F(u_0(x))d\sigma.
$$
Set
$$
\Psi(t)= \|u(t)\|^2_{L^2(\Om)}, \ \ t\geq 0,
$$
where $u(t)$ is a solution of the problem \eqref{nb1}-\eqref{nb3}. Then employing the equation \eqref{nb1}, the boundary condition \eqref{nb3} and the condition \eqref{Fc} we obtain
$$
\Psi^{\prime\prime}(t)\geq 2\|u_t(t)\|^2_{L^2(\Om)}-2\|\nabla u(t)\|^2_{L^2(\Om)}-2b \|u(t)\|^2_{L^2(\Om)}
+ 4(1+2\al)\int\limits_{\p \Om}F(u)d\sigma-2R_0,
$$
where $R_0= |\p \Om|r_0$.
Utilizing the energy equality \eqref{NBE} from the last inequality we obtain that
\begin{multline}\label{nb5}
\Psi^{\prime\prime}(t)\geq - 4(1+2\al)E_0 - 2R_0+4(1+\al)\|u_t(t)\|^2_{L^2(\Om)}\\
+4\al\|\nabla u(t)\|^2_{L^2(\Om)}+4\al b\|u(t)\|^2_{L^2(\Om)}.
\end{multline}
Employing  \eqref{nb5}, similar to the proof of the Theorem \ref{Levg} we can show that if
\be\label{nbb1}
\frac{(u_0,u_1)_{L^2(\Om)}}{\|u_0\|^2_{L^2(\Om)}}>2E_0+\frac{R_0}{1+2\al}>0,
\ee
then 
\be\label{nb6}
\Psi(t)=\|u(t)\|^2_{L^2(\Om)}\rw +\infty \ \ \mbox{as} \ \ t\rw +\infty.
\ee
Finally arguing as in the proof of the Theorem \ref{T1}  we get the inequality
$$
\Psi^{\prime\prime}(t)\Psi(t)-(1+\al)\left[\Psi^\prime(t)\right]^2 \geq \left(4\al b \Psi(t) -4(1+2\al)E_0-2R_0\right)\Psi(t).
$$
Thanks to the last inequality and the Lemma \ref{Le} we proved the following
\begin{theorem} \label{Tnb} If the conditions \eqref{nbb1} are satisfied, then the interval of existence $[0,T)$ of solution to the problem \eqref{nb1}-\eqref{nb3} is finite. Moreover
$$
\|u(t)\|^2_{L^2(\Om)}\rw +\infty \ \ \mbox{as} \ \ t\rw T^-.
$$
\end{theorem}
\begin{remark} Theorem \ref{Tnb} holds true also for the equation
$$
u_{tt}-\Delta u=0, \ \ x\in \Omega, t>0
$$
when a nonlinear boundary condition of the form\ 
$$
u(x,t)=0, \ \ x \in \Gamma_1, \ \ \frac{\p u}{\p n}=f(u), \ \ x\in \Gamma_2, t>0,
$$
where $\Gamma_1\cup \Gamma_2=\p \Omega, \ mes(\Gamma_1)\neq 0.$
\end{remark}

\end{document}